\documentclass[11pt]{article}

\usepackage[utf8]{inputenc}
\usepackage[T1]{fontenc}
\usepackage{mathpazo}

\RequirePackage[english]{babel}

\RequirePackage[a4paper,margin=2.5cm]{geometry}
\RequirePackage{parskip}

\newcommand{\affiliation}{\footnote}
\makeatletter
\def\@fnsymbol#1{\ensuremath{\ifcase#1\or *\or \dagger\or \ddagger\or \mathsection\or \|\or **\or \dagger\dagger \or \ddagger\ddagger \else\@ctrerr\fi}}
\makeatother

\usepackage{xcolor}
\definecolor{cblue}{RGB}{0,70,140}
\definecolor{cgreen}{RGB}{100,140,0}
\definecolor{cred}{RGB}{190,10,50}

\usepackage[shortlabels]{enumitem}
\setlist{itemsep=0ex,topsep=0ex,parsep=0.4ex}

\usepackage{amsmath,amsfonts,amssymb,amsthm,mathtools,thmtools,thm-restate,bbm,centernot}

\usepackage[hyphens]{url} % Line breaking of urls
\usepackage[hidelinks,colorlinks,pagebackref]{hyperref} % Coloring of hyperlinks and backreferences in bibliography; must be loaded after url
\hypersetup{citecolor=cgreen,linkcolor=cblue,urlcolor=cblue}
\usepackage[capitalise,nameinlink,noabbrev,compress]{cleveref} % References in the document; must be loaded after hyperref and amsmath

\renewcommand*{\backref}[1]{}
\renewcommand*{\backrefalt}[4]{
	\ifcase #1 Not cited.%
	\or $\uparrow$#2%
	\else $\uparrow$#2%
	\fi%
}

\theoremstyle{plain}
\newtheorem{theorem}{Theorem}[section]

\newtheorem{proposition}[theorem]{Proposition}

\theoremstyle{definition}
\newtheorem{definition}[theorem]{Definition}

\makeatletter
\renewenvironment{proof}[1][\proofname]
{\par\pushQED{\qed}
	\normalfont\topsep6\p@\@plus6\p@\relax\trivlist
	\item[\hskip\labelsep\bfseries#1\@addpunct{.}]
	\ignorespaces}
{\popQED\endtrivlist\@endpefalse}
\makeatother

\newcommand{\eps}{\varepsilon}

\newcommand{\cG}{\mathcal{G}}

\title{A note on Ramsey numbers for minors}
\author{Maria  Axenovich \affiliation{Institute of Algebra and Geometry, Karlsruhe Institute of Technology, Germany (\textsf{\href{mailto:maria.aksenovich@kit.edu}{maria.aksenovich@kit.edu}})}\and Raphael Steiner \affiliation{Department of Mathematics, ETH Zurich, Switzerland (\textsf{\href{raphaelmario.steiner@math.ethz.ch}{raphaelmario.steiner@math.ethz.ch}})}}
\date{\today}
  
\begin{document}
 \maketitle

  \begin{abstract}
  Let $R_h(k; \ell)$ be the smallest integer $n$ such that any edge coloring of a complete graph on $n$ vertices in $\ell$ colors results in a monochromatic $K_k$-minor, in other words, a graph with Hadwiger number $k$, i.e., a graph that could be transformed into a clique $K_k$ on $k$ vertices via a sequence of edge contractions and vertex deletions.  More generally, for a graph $F$ and integer $\ell$ let $R_h(F;\ell)$ be the smallest integer $n$ such that any edge coloring of a complete graph on $n$ vertices in $\ell$ colors results in a monochromatic $F$-minor.

In 2001 Thomason and in 2005  Myers and Thomason asymptotically determined the extremal numbers for clique minors and $F$-minors, respectively. They found the respective explicitly computable leading constants $\beta=0.265656...$ and $\gamma(F)\cdot \beta$ for these extremal numbers. 

\noindent We determine $R_h(F;2)$ for every graph $F$ as
  $$R_h(F;2)=(\gamma(F)+o(1))|V(F)|\sqrt{\log_2(|V(F)|)},$$ %where $\gamma(F)$ is an easily computable graph parameter defined by Myers and Thomason (Combinatorica, 2005) and 
  where the $o(1)$-term tends to zero as $|V(F)|\rightarrow \infty$. In particular,
  $$R_h(k;2)=(1+o(1))k\sqrt{\log_2 k}.$$
  
  When $\ell\gg k \gg 1$, we show that 
  $$  R_h(k; \ell) =  (2\beta+o(1))   \ell k \sqrt{\log_2 k}.$$ 
  %where $\beta=0.265656...$ is a specific constant.
  \end{abstract}
  
  \section{Introduction}
 
 The classical {\it Ramsey number}  $R(k)$ is defined as the smallest integer $n$ such that any edge coloring of a complete graph $K_n$ on $n$ vertices with two colors contains a monochromatic $K_k$. This notion was generalised in many different ways, in particular ``through parameter''. Specifically, when $\xi$ is some graph parameter, then $R_{\xi}(k)$ is defined as the smallest $n$ such that any edge coloring  of $K_n$ in two colors contains a monochromatic subgraph $G$ with $\xi(G)=k$.  So, $R(k)=R_\omega(k)$, where $\omega(G)$ is the clique number of $G$.  One of the classical graph parameters is the {\it Hadwiger number} $h(G)$ which is the largest $k$ so that $G$ contains a $K_k$-minor, that means that $K_k$ can be obtained from $G$ by edge-contractions and vertex-deletions.  In this note, we shall discuss the Ramsey number for minors $R_h(k)$.  Although it is very natural and complements extensive research on Hadwiger's conjecture, it seems that it never was explicitly addressed in the literature. The Hadwiger conjecture \cite{H} states  that  for any graph $G$, $h(G)\geq \chi(G)$, where $\chi(G)$ is the chromatic number of $G$, see for example Delcourt and Postle \cite{DP}, see also~\cite{Seymour} for a survey.   A related density question is to bound $h(G)$ in terms of the average degree of $G$.  For each $k\geq 2$, let
  $$c(k):= {\rm min} \{c: |E(G)|\geq c|V(G)| \Rightarrow h(G)\geq k \}.$$
  The fact that one could take minimum and not just an infimum in the above definition is due to Mader \cite{M-0}.
   Thomason  \cite{T} determined $c(k)$ asymptotically, see some related papers \cite{AKS, T-two-problems, T-survey, TW}.     
   Let $\lambda<1$ be the solution of the equation 
$1-\lambda +2\lambda \ln \lambda =0$ and 
    and $\beta = (1-\lambda)/2\sqrt{\ln (1/\lambda)}/\sqrt{\log_2e}$. Here, $\beta =0.3190863.../\sqrt{\log_2 e} = 0.265656...$.
   
   \begin{theorem}[Thomason \cite{T}]\label{Thomason-1}
  We have  $$c(k) = (1+o(1))\beta k \sqrt{\log_2 k}.$$
   \end{theorem}

The lower bound on $c(k)$ in the above theorem is achieved by the Erd\H{o}s-Renyi random graph $G(n,p)$, whose Hadwiger number is well understood.      
 \begin{theorem}[Bollob\'as, Catlin, and Erd\H{o}s \cite{BCE}]\label{BCE}
   For a fixed $p$ and $q=1-p$, 
    $$h(G(n,p)) = (1+o(1))n \left(\frac{\log (1/q)}{ \log n} \right)^{1/2}.$$
    \end{theorem}

In addition, Thomason \cite{T}  proved that  the Hadwiger number for graphs whose connectivity is not too small is at least one  of a random graph with the same density. Here, the density of a graph $G$ is $|E(G)|/\binom{|V(G)|}{2}$.
\begin{theorem}[Theorem 4.1, Thomason \cite{T}]\label{Thomason-2}
Let $0<\varepsilon<1$. There is $n_0=n_0(\varepsilon)$, such that if $G$ is a graph on $n$ vertices, $n>n_0$,   of density $p$,  and connectivity at least $n (\log \log \log n)/ (\log \log n)$, then  
$$h(G)\geq \left\lceil (1-\varepsilon) n \left(  \frac{\log (1/q^*)}{\log n } \right)^{1/2}\right\rceil,$$ 
where $q^*= \max\{ 1-p, (\log n)^{-1/\varepsilon}\}$.
\end{theorem}

Each of Theorems~\ref{Thomason-1}, \ref{BCE}, \ref{Thomason-2} have been generalized to arbitrary minors $F$ by Myers and Thomason~\cite{MyTh}, who in particular determined $c(F)$, where 
$$c(F):= \inf \{c: |E(G)|\geq c|V(G)| \Rightarrow G \mbox{ has an } \mbox  {$F$-minor}\}.$$
To state these more general results, we need to introduce a parameter $\gamma(F)\in [0,1]$ associated with any graph which was defined by Myers and Thomason~\cite{MyTh} as the optimal value of an optimization problem over vertex-weightings as follows.

\begin{definition}

Let $F$ be a graph on $k$ vertices. Then we define

$$\gamma(F)=\min\left\{\frac{1}{k}\sum_{v\in V(F)}w(v)\bigg\vert w:V(F)\rightarrow [0,\infty) \text{ and }\sum_{uv\in E(F)}k^{-w(u)w(v)}\le k\right\}.$$

\end{definition}

Myers and Thomason~\cite{MyTh}  determined $\gamma(H)$ precisely for several interesting graphs. In particular, they showed that $\gamma(K_k)=1$ for every $k$, gave an explicit formula for $\gamma(F)$ for every complete multi-partite graph $F$, showed that every graph $F$ on $k$ vertices with average degree $k^\tau$ satisfies $\gamma(F)\le \sqrt{\tau}$, and most interestingly that almost all graphs $F$ on $k$ vertices with this average degree satisfy $\gamma(F)=\sqrt{\tau}-o(1)$. 

With this notation at hand, we can now state the aforementioned generalizations of Theorems~\ref{Thomason-1}, \ref{BCE}, \ref{Thomason-2} to arbitrary minors.

\begin{theorem}[Theorem 2.2, Myers and Thomason \cite{MyTh}]\label{thm:generaldensity}
Let $F$ be a graph on $k$ vertices. Then %Then the extremal density $c(F)$ for $F$-minors, defined as the smallest value $c$ such that every non-null graph $G$ with $|E(G)|\ge c|V(G)|$ contains $F$ as a minor, satisfies
$$c(F)=(\gamma(F)\beta+o(1))k\sqrt{\log_2 k},$$
where the o(1)-term tends to 0 as $k\rightarrow \infty$. 
\end{theorem}

\begin{theorem}[Theorem 2.3, Myers and Thomason\cite{MyTh}]\label{thm:generalrandom}
For every $\varepsilon>0$ there exists $k_0=k_0(\varepsilon)$ such that for every $k\ge k_0$, every graph $F$ with $k$ vertices satisfying $\gamma(F)\ge \varepsilon$, every $p\in [\varepsilon, 1-\varepsilon]$, and for every number $n\le \gamma(F)k\sqrt{\log_{1/(1-p)}(k)}$, we have that $F$ is a minor of $G(n,p-\varepsilon)$ with probability less than $\varepsilon$.  
\end{theorem}

\begin{theorem}[Theorem 2.5, Myers and Thomason \cite{MyTh}]\label{thm:generalconnected}
For every $\varepsilon>0$ there exists $k_0=k_0(\varepsilon)$ such that for every $k\ge k_0$, every graph $F$ with $k$ vertices satisfying $\gamma(F)\ge \varepsilon$, every $p\in [\varepsilon, 1-\varepsilon]$, and for every number $n\ge \gamma(F)k\sqrt{\log_{1/(1-p)}(k)}$, the following holds.
If $G$ is a graph on $n$ vertices, with density $p+\varepsilon$ and connectivity at least $n(\log \log \log n)/(\log \log n)$, then $F$ is a minor of $G$.
\end{theorem}

Kostochka \cite{K}, Stiebitz \cite{St}, as well as Reed and Thomas \cite{RT}  studied the Hadwiger number for graphs and their complements in  an ``off-diagonal'' setting when $h(G)$ is a fixed constant. 
For the rest of this note, the logarithms are base $2$, i.e., $\log = \log_2$ and $\beta$ is a constant defined above. We drop floors and ceilings where clear from context.   

 Let $R_h(k; \ell)$ be the smallest integer $n$ such that any edge coloring of $K_n$ in $\ell$ colors results in a monochromatic graph with Hadwiger number $k$ (we also abbreviate $R_h(k):=R_h(k;2)$). More generally, for a graph $F$ let us denote by $R_h(F;\ell)$ the smallest $n$ such that every $\ell$-edge coloring of $K_n$ results in a monochromatic $F$-minor. The following theorems leverage the previously stated results of Myers and Thomason to asymptotically determine $R_h(F;2)$ up to lower-order terms for all graphs $F$ and in particular $R_h(k)$, and also to asymptotically determine $R_h(k;\ell)$ when $\ell$ is much larger than $k$. 

   \begin{theorem} \label{main}
For every graph $F$ on $k$ vertices we have $$R_h(F;2)=(\gamma(F)+o(1))k\sqrt{\log k}$$ where the $o(1)$-term tends to zero as $k\rightarrow \infty$. 
In particular, we have  $$R_h(k)=(1+o(1))k \sqrt{\log k}.$$ 
Moreover,  $R_h(3)=5$, $R_h(4)=7$, and for any $k\geq 5$, 
$2k-2 \leq R_h(k) \leq  32(k-2)\lfloor \log(k-2) \rfloor +1.$
  \end{theorem}
  
\begin{theorem}\label{multicolor}
Let $\beta = 0.2656568...$ as above. For any $\ell\geq 1$,  $R_h(3; \ell) = 2\ell+1$.
 For any $\varepsilon>0$ there is $k_0=k_0(\varepsilon)$, such that for any $k\geq k_0$ there is $\ell_0=\ell_0(k, \varepsilon)$, so that for any $\ell>\ell_0$, 
$$ 2\beta \ell k \sqrt{\log k}  (1- \varepsilon)  \leq R_h(k; \ell) \leq  2\beta \ell k \sqrt{\log k} (1+\varepsilon).$$
\end{theorem}
  
  \section{Proofs of the main results}
  
  \begin{proof}[Proof of Theorem \ref{main}]
  
 %First, we note that we can get an upper bound $R_h(k) \leq   4\beta k \sqrt{\log k} (1+o(1)) \approx 1.0623 k \sqrt{\log k} (1+o(1))$ from Theorem \ref{Thomason-1} immediately.  

We start by proving the first part of the theorem concerning the asymptotic behavior of $R_h(F;2)$, which then also implies the asymptotic behavior of $R_h(k)$, since $R_h(k)=R_h(K_k;2)$ and $\gamma(K_k)=1$. 

Formally, we need to prove that for every $\varepsilon>0$ and every sufficiently large number $k$ (in terms of $\varepsilon$), every graph $F$ on $k$ vertices satisfies (denoting $\gamma=\gamma(F)$ in the following):

$$(\gamma-\varepsilon)k\sqrt{\log_2 k}\le R(F;2)\le (\gamma+\varepsilon)k\sqrt{\log_2 k}.$$

Suppose first that $\gamma\le \eps/2$. Then the claimed lower bound holds trivially for every $k$.
To see that the upper bound holds for every sufficiently large $k$ in this case, consider a coloring of the edges of $K_n$ in two colors for $n\ge \lfloor(\gamma+\varepsilon)k\sqrt{\log_2 k}\rfloor$.  A larger of the two color classes  corresponds to a graph $G$ satisfying $|E(G)|/|V(G)|\ge (n(n-1)/4)/n=(n-1)/4\ge (\beta\gamma+\varepsilon/10)k\sqrt{\log_2 k}$ for every large enough $k$, where the last inequality holds since 
$\gamma+\varepsilon\ge 4(\beta \gamma+\varepsilon/8)$. 
%Hence, for any coloring of the edges of $K_n$ in two colors for $n\ge \lfloor(\gamma(F)+\varepsilon)k\sqrt{\log_2 k}\rfloor$, the larger of the two color classes will correspond to a graph $G$ satisfying $|E(G)|/|V(G)|\ge (n(n-1)/4)/n=(n-1)/4\ge (\beta\gamma(F)+\varepsilon/10)k\sqrt{\log_2 k}$ for every large enough $k$. 
Hence, by Theorem~\ref{thm:generaldensity}, $G$ contains $F$ as a minor given that $k$ is sufficiently large, as desired.

Moving on, we may thus assume that $\gamma \ge \eps/2$. Furthermore, we may trivially assume w.l.o.g. that $\varepsilon<1/2$.

We first prove the lower bound on $R_h(F;2)$. To do so, we consider a random coloring of $K_n$, where $n=\lfloor(\gamma -\varepsilon) k \sqrt{\log k}\rfloor$, and where each edge is independently colored red with probability $1/2$ and blue with probability $1/2$. Each color class corresponds to $G(n, 1/2)$. 

Choose $\varepsilon^\ast>0$ (only dependent on $\varepsilon)$ small enough such that $\varepsilon^\ast< \varepsilon/2$, $1/2+\varepsilon^\ast\le 1-\varepsilon$, and $1/\sqrt{\log_2(1/(1/2-\varepsilon^\ast))}\ge 1-\varepsilon$. Now, let us set $p^\ast:=1/2+\varepsilon^\ast$ and note that by our choice of $\varepsilon^\ast$, we have $\gamma \ge \varepsilon/2\ge \varepsilon^\ast$, $p^\ast\in [\varepsilon^\ast,1-\varepsilon^\ast]$ and $$n\le (\gamma -\varepsilon)k\sqrt{\log_2 k}\le (1-\varepsilon)\gamma k\sqrt{\log_2 k} \le \frac{\gamma k\sqrt{\log _2 k}}{\sqrt{\log_2(1/(1/2-\varepsilon^\ast))}}=\gamma k\sqrt{\log_{1/(1-p^\ast)}k}.$$ Hence, the preconditions of Theorem~\ref{thm:generalrandom} with parameters $\varepsilon^\ast, p^\ast$ are met, and we find for every sufficiently large $k$ that $G(n,1/2)=G(n,p^\ast-\varepsilon^\ast)$ contains $F$ as a minor with probability less than $\varepsilon^\ast<1/2$. Hence, with positive probability, both color classes of our random edge-coloring of $K_n$ do not contain $F$ as a minor, showing that $R_h(F;2)\ge n+1>(\gamma -\varepsilon)k\sqrt{\log_2 k}$. This concludes the proof of the desired lower bound. 

Next, let us prove the upper bound. So let $k$ be sufficiently large in terms of $\varepsilon$ (to be specified later) and let a coloring of the edges of $K_n$ with colors red and blue be given to us, where $n= \lfloor (\gamma(F)+\varepsilon)k\sqrt{\log_2 k}\rfloor$. Let $G$ denote the graph on the same vertex set as $K_n$ formed by the edges of a color class with the most edges, w.l.o.g. let us say that this color class consists of the red edges. Let $\varepsilon'>0$ be chosen sufficiently small, only depending on $\varepsilon$, such that $\varepsilon'< \min\{\varepsilon/8,1/100\}$ and such that $\frac{1+10\varepsilon'}{\log_2\left(2/(1+10\varepsilon')\right)}<(1+\varepsilon/2)^2.$

Let now $H$ denote the $\lceil \varepsilon' n\rceil$-core of $G$, i.e., the 
(unique) maximal induced subgraph of $G$ with minimum degree at least $\varepsilon' n$.  To see that $H$ is well-defined, recall that for every $d\in\mathbb{R}^+$, every graph of average degree at least $d$ contains a non-empty subgraph of minimum degree at least $d/2$. Since $G$ contains at least half of the edges of $K_n$, it has average degree at least $(n-1)/2$ and hence contains a subgraph with minimum degree at least $(n-1)/4>\varepsilon' n$ for $k$ sufficiently large. 

Note that $H$ can be obtained from $G$ by repeatedly deleting vertices of degree less than $\varepsilon' n$. Hence, we have $|E(H)|\ge |E(G)|-\varepsilon' n^2\ge n(n-1)/4-\varepsilon' n^2\ge (1/4-2\varepsilon')n^2$ for $k$ large enough. This in particular implies that $x:=|V(H)|\ge \sqrt{2|E(H)|}\ge \sqrt{1/2-4\varepsilon'}\cdot n$. In the following, let $\alpha:=x/n\in \left[\sqrt{1/2 -4\varepsilon'},1\right]$ and let $p:=|E(H)|/\binom{x}{2}$ denote the density of $H$. Note that from the above, we immediately obtain that $$p\ge (1/4-2\varepsilon')n^2/\binom{x}{2}\ge (1/2-4\varepsilon')(n/x)^2
= 1/(2\alpha^2)-(4\varepsilon')/\alpha^2
\ge 1/2-4\varepsilon'.$$
Moving on, we now distinguish two cases.

\textbf{Case~1.} $H$ has vertex-connectivity at least $n(\log \log \log n)/(\log \log n)$.

So, $H$ has $x$ vertices, $x\ge \sqrt{1/2-4\varepsilon'}\cdot n$, it  has connectivity at least $x(\log \log \log x)/(\log \log x)$, and it has density $p\ge 1/2-4\varepsilon'\ge 2\varepsilon'$ (by our choice of $\varepsilon'$). We shall show that $H$ satisfies the conditions of Theorem \ref{thm:generalconnected} and thus has an $F$-minor.

Let $p':=p-\varepsilon'$. Then $p'\in [\varepsilon',1-\varepsilon']$. Also, recall that $\gamma\ge \varepsilon/2\ge \varepsilon'$. 

Next, we claim that $x\ge \gamma k\sqrt{\log_{1/(1-p')}(k)}$ for large enough $k$. Indeed, we have

$$x=\alpha n\ge \alpha\left(\gamma+\frac{\varepsilon}{2}\right)k\sqrt{\log_2(k)}\ge \alpha\gamma\left(1+\frac{\varepsilon}{2}\right)k\sqrt{\log_2(k)}
=\frac{(1+\frac{\varepsilon}{2})\alpha}{\sqrt{\log_{(1/(1-p'))}2}}\cdot \gamma k\sqrt{\log_{1/(1-p')}(k)}.$$

It remains to show that $(1+\varepsilon/2)\alpha\sqrt{\log_2(1/(1-p'))}\ge 1$. To do so, we start by recalling that $p'=p-\varepsilon'\ge \frac{1}{2\alpha^2}-\frac{4\varepsilon'}{\alpha^2}-\varepsilon'$. Since furthermore $\alpha\ge \sqrt{\frac{1}{2}-4\varepsilon'}$ and $\varepsilon'<1/100$, we find that in fact
$p'\ge \frac{1}{2\alpha^2}-\frac{4\varepsilon'}{\frac{1}{2}-4\varepsilon'}-\varepsilon'\ge \frac{1}{2\alpha^2}-10\varepsilon'$. Rearranging yields $\alpha\ge \frac{1}{\sqrt{2p'+20\varepsilon'}}$, and so we obtain

$$(1+\varepsilon/2)\alpha\sqrt{\log_2(1/(1-p'))}\ge (1+\varepsilon/2) \sqrt{\frac{\log_2(1/(1-p'))}{2p'+20\varepsilon'}}.$$
Since $p'\ge 1/2-5\varepsilon'$, we furthermore have $$2p'+20\varepsilon'\le 2p'\left(1+\frac{10\varepsilon'}{1/2-5\varepsilon'}\right)=2p'\frac{1+10\varepsilon'}{1-10\varepsilon'}.$$ Plugging this in, it follows that
$$(1+\varepsilon/2)\alpha\sqrt{\log_2(1/(1-p'))}\ge (1+\varepsilon/2) \sqrt{\frac{\log_2(1/(1-p'))}{2p'}} \cdot \sqrt{\frac{1-10\varepsilon'}{1+10\varepsilon'}}.$$ It is not difficult to check that the real-valued function $f:(0,1)\rightarrow \mathbb{R}_+$, $f(y):=\frac{\log_2(1/(1-y))}{2y}$ is monotonically increasing in $y$. In particular, $f(p')\ge f(1/2-5\varepsilon')$, yielding
$$(1+\varepsilon/2)\alpha\sqrt{\log_2(1/(1-p'))}\ge (1+\varepsilon/2)\sqrt{\frac{\log_2\left(2/(1+10\varepsilon')\right)}{1+10\varepsilon'}}\ge 1,$$ where we used our choice of $\varepsilon'$ in the last step. This is the desired inequality, and so we have shown that indeed, $H$ has $x\ge \gamma k\sqrt{\log_{1/(1-p')}(k)}$ vertices for large enough $k$. 

Note that $H$ has density $p=p'+\varepsilon'$. Hence, we have now met all preconditions of Theorem~\ref{thm:generalconnected} with $x$ in place of $n$, $\varepsilon'$ in place of $\varepsilon$, $p'$ in place of $p$, and $H$ in place of $G$. It thus follows that whenever $k$ (and thus $n$ and hence $x$) are sufficiently large in terms of $\varepsilon$, the graph $H$ contains $F$ as a minor. In particular, the supergraph $G$ consisting of the red edges of the coloring of $K_n$ we initially considered contains $F$ as a minor. This is what we wanted to show, and concludes the proof in Case~1.

\textbf{Case~2.} $H$ has vertex-connectivity less than $n(\log \log \log n)/(\log \log n)$. 

This means that there exists a subset of vertices $S\subseteq V(H)$ with $|S|<n(\log \log \log n)/(\log \log n)$ and a partition $(A,B)$ of $V(H)\setminus S$ into two non-empty parts such that there are no edges in $H$ connecting $A$ and $B$. Pick any vertices $a\in A$ and $b\in B$. Then all neighbors of $a$ lie in $A\cup S$, and all neighbors of $b$ lie in $B \cup S$. Since $H$ by definition has minimum degree at least $\varepsilon'n$, it follows that $|A\cup S|, |B\cup S|\ge \varepsilon'n$, and hence that
$$\min\{|A|,|B|\}\ge \varepsilon' n-|S|>\left(\varepsilon'-\frac{\log\log\log n}{\log\log n}\right)n\ge k$$ for $k$ sufficiently large.

Observe that all edges of our initial coloring of $K_n$ spanned between $A$ and $B$ must be colored blue, since they cannot be contained in $G$. It follows that our initial coloring of $K_n$ contains a blue $K_{k,k}$ subgraph. Since $K_{k,k}$ contains $K_k$-minor, and hence also $F$-minor, we find that in Case~2 (provided $k$ is sufficiently large), the blue subgraph of our colored $K_n$ contains $F$ as a minor. This concludes the proof in Case~2. 
 
This concludes the proof of the desired upper bound on $R_h(F;2)$ and thus of the first part of the theorem.  
  
  Now we prove the bounds for small $k$ and give general, but not optimal bounds that hold for any $k$.
  To see the lower bound of $2k-2$, consider a  union of two red vertex-disjoint cliques, one on $k-1$ and the other on $k-2$ vertices, and color all other edges in between them blue.  
The upper bound follows from Theorem 1.16 by Bollob\'as \cite{B}, that claims that 
  $c(k) \leq 8(k-2)\lfloor \log_2 (k-2) \rfloor$. Indeed, let $G$ be a majority color class in a given red/blue coloring of the edges of $K_n$ with $n=32(k-2)\lfloor \log(k-2)\rfloor+1$. It has at least $n(n-1)/4  $ edges. Since  $(n-1)/4 \geq 8(k-2)\lfloor \log_2 (k-2) \rfloor$, we have that $h(G)\geq k$.

  Finally, we treat the small cases $k\in \{3,4\}$.   To see that $R_h(3)=5$, consider an arbitrary edge coloring of $K_5$ in two colors. A largest color class has at least $5$ edges and thus has a cycle, that is a $K_3$-minor. On the other hand, a coloring of $K_4$ where each color class is isomorphic to $P_4$ contains no cycles and thus no $K_3$-minors.
  
 For $R_h(4)$, we shall use the fact that the largest number of edges in an $n$-vertex graph with no $K_4$-minor is at most  $2n-3$, for $n\geq 3$, see for example Bollob\'as \cite{B}.   To see that $R_h(4)>6$, consider a 6-cycle and three additional new edges forming a triangle, color this graph red and  its complement blue. Since any $K_4$-minor has at least four vertices of degree at least three, we see that the color classes are $K_4$-minor free.  We can immediately see that $R_h(4) \leq 8$.  Indeed,  in any  2-edge coloring of $K_8$,   a majority color class  has size at least   $14 > 2\cdot 8 -3$ and thus contains a $K_4$-minor.  Now, we shall show a stronger bound $R_h(4)\leq 7$. Consider an arbitrary red/blue edge coloring of $K_7$ with no monochromatic $K_4$-minor  and let $G$ be a graph whose  edge set consists of the red edges.  We make several observations.  As above  $10\leq |E(G)|\leq 11$. Both $G$ and $\overline{G}$  are not bipartite since otherwise one of the parts has size at least $4$. By a special case of Hadwiger's conjecture, $\chi(G)= \chi(\overline{G})=3$.  Thus $G$ and $\overline{G}$  are three-partite graphs with parts of sizes at most three, i.e., of sizes $3,3,1$ or $3, 2, 2$. There is no monochromatic $K_{3,3}$ without an edge. If one color class, say blue,  has two vertex disjoint triangles, then there are two blue edges between them or three edges forming a $P_4$ between them. 
By considering the remaining vertex, one can see that there is a monochromatic copy of $K_4$-minor, a contradiction. Both the red and the blue graphs are three-chromatic with color classes of sizes $3,2,2$. This leaves a final, not too long, case analysis to demonstrate the existence of a monochromatic $K_4$-minor, a contradiction.
 \end{proof}

\begin{proof}[Proof of Theorem \ref{multicolor}]
To see that $R_h(3; \ell) \geq 2\ell+1$, consider a decomposition of $K_{2\ell}$ into $\ell$ Hamiltonian paths. Since any $K_3$-minor contains a cycle, we have that this coloring has no monochromatic $K_3$-minors. 
To see that $R_h(3; \ell) \leq 2\ell+1$, consider an arbitrary edge-coloring of $K_{2\ell+1}$ in $\ell$ colors. Then some color class has at least $2\ell +1$ edges and thus contains a cycle, that is a $K_3$-minor.

The upper bound of the theorem holds in a stronger form, for any $\ell$. Let $n=R_h(k; \ell)-1$ and consider an edge-coloring of $K_n$  in $\ell$ colors with no monochromatic $K_k$-minor. 
From Theorem \ref{Thomason-1}, we know that any $K_k$-minor free graph on $n$ vertices has at most $\beta k \sqrt{\log k} n (1+o_k(1))$ edges. 
 Since the total number of edges on $n$ vertices is $\binom{n}{2}$, we have that 
$\binom{n}{2} \leq  \ell \cdot \beta k \sqrt{\log k} n (1+o_k(1))$. This implies that $n \leq  \ell \cdot 2\beta k \sqrt{\log k}(1+o_k(1))$.

For the lower bound,  let $\varepsilon$ be a  small constant,  $0<\varepsilon <1$,  $k$ be sufficiently large.  Consider a graph $H'$ on $k\sqrt{\log k}$ vertices and $m=\beta k^2\log k (1- \varepsilon^2/5)$ edges that is $K_k$-minor free, that exists by Theorem \ref{Thomason-1}.  Let $H$ be obtained from $H'$ by attaching a pendant edge. 
 Wilson \cite{W} showed that $K_n$ is a pairwise edge-disjoint union of copies of  $H$ if $n$ is sufficiently large, $|E(H)|$ divides $\binom{n}{2}$, and the greatest common divisor of the degrees of $H$ divides $n-1$. Since there is a vertex of degree $1$ in $H$, the last condition is satisfied for any $n\geq 2$. 
 Let $\ell$ be sufficiently large, so that for some $n$ satisfying $\ell \cdot 2\beta k \sqrt{\log k} (1-\varepsilon) \leq n \leq  \ell \cdot 2\beta k \sqrt{\log k} (1-\varepsilon/2)$, $n$ satisfies conditions of Wilson's  theorem, in particulst $|E(H)|$ divides $\binom{n}{2}$.
 Then $H$ decomposes $K_n$. 
  Alon and Yuster \cite{AY}  proved that if  a fixed graph  $H$ decomposes $K_n$, then there is a coloring of $E(K_n)$ in at most $(1+o(1)) (n-1)|V(H)|/(2|E(H)|)$ colors, such that each color class is a pairwise  vertex-disjoint union of copies of $H$. So, we take $\ell$ sufficiently large that $o(1)$ term is at most $\varepsilon/2$. That is, every color class is almost an $H$-factor and it is $K_k$-minor free.  Since the number of colors is at most $(1+\varepsilon/2) (n-1)|V(H)|/(2|E(H)|) \leq  (1+\varepsilon/2)\ell \cdot 2\beta k \sqrt{\log k} (1-\varepsilon/2) /(2\beta k \sqrt{\log k}(1-\varepsilon^2/5)) \leq \ell$, we have a 
  coloring of $E(K_n)$ in $\ell$ colors with no monochromatic copy of a $K_k$-minor. 
\end{proof}

\section{Concluding remarks}
We determined the asymptotic value of $R_h(k; \ell)$ when $\ell=2$ and when $k\gg \ell\gg 1$. It would be interesting to extend these results, and in particular asymptotically determine the value of $R_h(k;\ell)$ as $\ell\ge 3$ is fixed and $k\rightarrow \infty$. It would also be nice to find an easy direct proof of an upper bound on $R_h(k; 2)\leq c k\sqrt{\log k}$, for some, maybe large  constant $c$,  but avoiding the density Theorem \ref{Thomason-1}.

To prove the lower bound on $R_h(k; \ell)$, one can also use the following packing result by taking $G$ being an $H$-factor, where $H$ is a graph on $k\sqrt{\log k}(1+o(1))$ vertices, density $\beta k \sqrt{\log k}$ and no $K_k$-minor.
Pack  $K_{n(1-\varepsilon)}$ with copies of $G$, give each such copy of $G$ its own color. Decompose the remaining edges of $K_n$ into forests using the Nash-Williams theorem and give each such forest a new own color. 
 \begin{theorem}[Messuti, R\"odl, Schacht \cite{MRS}]\label{MRS}
For any $\varepsilon>0$ and any $\Delta\in \mathbb{N}$, for any non-trivial minor closed family $\cG$ there is $n_0$ such that for any $n\geq n_0$ and any integer $\ell$ the following holds.
If $G\in \cG$, $\Delta(G) \leq \Delta$ and $|V(G)|\leq n$,  and $\ell |E(G)| \leq \binom{n}{2}$ then $\ell$ copies of $G$ pack in $K_{(1+\varepsilon) n}$.
\end{theorem}

Note that our Theorem~\ref{main} nails down the value of $R_h(F;2)$ precisely for every graph $F$ with $\gamma$ bounded away from zero, and thus for almost all graphs on $k$ vertices with at least $k^{1+\varepsilon}$ edges for any fixed $\varepsilon>0$~\cite{MyTh}. However, in the sparse regime, i.e., when $F$ has only $k^{1+o(1)}$ edges, it only says that $R_h(F;2)=o(k\sqrt{\log k})$. Thus, it would be interesting to determine $R_h(F;2)$ more precisely for very sparse graphs, with, say, a constant average degree. A natural first step is to consider $R_h(F)$ for trees $F$ on $k$ vertices. In the case that $F=P_k$ is a path on $k$ vertices, we simply have $R_h(P_k;2)=R(P_k;2)=\left\lfloor \frac{3k-2}{2} \right\rfloor$ for every $k$, since any graph contains $P_k$ as a minor if and only if it contains $P_k$ as a subgraph, and since the Ramsey number of a path was determined to equal the above value by Gerencs\'{e}r and Gy\'{a}rf\'{a}s~\cite{gyarfas}. Moreover, also for multiple colors the known results on the Ramsey numbers of paths extend to their minor-Ramsey numbers. On the other extreme, for a star $S_k=K_{1,k}$ on $k+1$ vertices, we show that $R_h(S_k;2)$ is by a factor two smaller than $R(S_k;2)=2k+O(1)$.
\begin{proposition}\label{prop:stars}
We have $R_h(S_k;2)=k+o(k)$.
\end{proposition}
\begin{proof}
Since no graph on at most $k$ vertices contains $S_k$ a minor, we have $R_h(S_k;2)\ge k+1$, and hence it suffices to
prove the upper bound. So fix $\varepsilon>0$ arbitrarily and let us show that for every sufficiently large $k$, every red/blue-coloring of the edges of $K_{\lceil(1+\varepsilon)k\rceil}$ contains a monochromatic $S_k$-minor. In the following, we set $n:=\lceil(1+\varepsilon)k\rceil$. Consider any red/blue-coloring of the edges of $K_n$ and let $G_r,G_b$ denote the spanning subgraphs of $K_n$ formed by the red and blue edges, respectively.

If any vertex in any of $G_r, G_b$ has degree at least $k$, then our coloring contains $S_k$ as a monohromatic subgraph and we are done. Hence, in the following we may assume that $\Delta(G_r), \Delta(G_b)<k$. Since $G_r$ and $G_b$ are complements of each other, this in particular implies that $G_r$ and $G_b$ have minimum degree at least $(n-1)-(k-1)\ge \lceil\varepsilon k\rceil=:s$ each.

Now, let $S\subseteq V(K_n)$ be a random subset of vertices created as follows. Randomly pick vertices $v_1,\ldots,v_s$ from $V(K_n)$ with repetition, i.e. such that $v_i$ is uniformly random in $V(K_n)$ independently from all other $v_j$. Finally, we set $S:=\{v_1,\ldots,v_s\}$. We have $|S|\le s$ and hence $|V(K_n)\setminus S|\ge n-s=k$. 

We now claim that for every sufficiently large $k$, the set $S$ simultaneously forms a dominating set in both of the graphs $G_r,G_b$. Indeed, for a fixed vertex $v\in V(K_n)$ and a fixed color $i\in \{r,b\}$, the probability that $v$ has no neighbors in $S$ in the graph $G_i$, i.e., that none of the vertices $v_1,\ldots,v_s$ was picked from $N_{G_i}(v)$, is 
$$\left(1-\frac{d_{G_i}(v)}{n}\right)^s\le \exp\left(-\frac{d_{G_i}(v)s}{n}\right)\le \exp\left(-\frac{s^2}{n}\right)=\exp(-\Theta(k)).$$
Hence, the probability that some vertex $v$  is not dominated by $S$ in $G_i$ for some color $i\in\{r,b\}$ is at most $n\cdot 2\cdot \exp(-\Theta(k))=O(k)\cdot \exp(-\Theta(k))\rightarrow 0$ for $k\rightarrow \infty$. Hence, indeed for every sufficiently large $k$, with positive probability, the set $S$ is a dominating set of both $G_r$ and $G_b$. 

Now, assuming $k$ is sufficiently large, fix a set $S$ of size at most $s$ which dominates $G_r$ and $G_b$. Since $G_r[S]$ and $G_b[S]$ are complementary graphs on vertex set $S$, at least one of them, say $G_r[S]$, must be connected. But then we obtain $S_k$ as a minor of $G_r$ by contracting $S$ (more precisely, all edges contained in it) into a single vertex. This is the desired outcome and concludes the proof.
\end{proof}
 
\section{Acknowledgements} The first author is grateful to Noga Alon for several very fruitful discussions and suggestions. 
She also thanks Jan Malte Weirich for discussions.

The second author gratefully acknowledges funding from the Ambizione grant No. 216071 of the Swiss National Science Foundation.

\end{document}